\documentclass[11pt,a4paper]{article}
\usepackage[pdftex]{graphicx}
\usepackage{amsmath,amssymb,amsfonts,amsthm}
\numberwithin{equation}{section}
\usepackage{indentfirst}
\usepackage{enumitem} 
\usepackage[margin=1.3in]{geometry}
\usepackage{fancyhdr}
\usepackage{float}

\usepackage[colorlinks=true,linkcolor=magenta,citecolor=blue,urlcolor=cyan]{hyperref}
\usepackage{titlesec}
\usepackage{etoolbox}
\usepackage{graphicx}  
\usepackage{changepage}
\theoremstyle{plain}
\newtheorem{theorem}{Theorem}[section]
\newtheorem{lemma}[theorem]{Lemma}
\newtheorem{corollary}[theorem]{Corollary}

\newtheorem{definition}[theorem]{Definition}
\newtheorem{remark}{Remark}[section]

\makeatletter
\renewcommand{\maketitle}{
	\begin{center}
		{\Large\bfseries{\@title}\par}
		\vskip 1em
		{\normalsize
			\lineskip .5em
			\begin{tabular}[t]{c}
				\@author
			\end{tabular}\par}
		\vskip 1.5em
	\end{center}
}
\makeatother
\renewenvironment{abstract}{
	\begin{adjustwidth}{1.3cm}{1.3cm}
		\noindent{\large\bfseries{A{\scriptsize BSTRACT.}}}
	}{
	\end{adjustwidth}
}

\usepackage{color}
\usepackage{xcolor}
\usepackage[normalem]{ulem} 
\usepackage{soul}

\usepackage{graphicx} 
\usepackage{xstring}  

\usepackage{xstring}
\usepackage{graphicx} 

\usepackage{marvosym}

\begin{document}
	\title{A Note On Certain Minimal Excludants Over Overpartitions}
	\author{Dipika Sarkar, M. P. Thejitha, and S. N. Fathima}
	\maketitle
	\begin{abstract}
		Let $\sigma\mathrm{Mex}(n)$ and $\sigma_e\mathrm{Mex}(n)$ denote the sum of minimal excludants and sum of even minimal excludants over all overpartitions, respectively. In this work, we study these two combinatorial objects from an arithmetic perspective. We prove that, for $n\geq 0$, 
		\begin{equation*}
			\sigma_e\mathrm{Mex}(n)\equiv\sigma\mathrm{Mex}(n)-\bar{p}_{\geq 2}(n)\pmod{2^2},
		\end{equation*}
		where $\bar{p}_{\geq 2}(n)$ denotes the number of overpartitions of $n$ with all parts at least $2$, and obtain relation with basic hypergeometric series.  Furthermore, we prove the asymptotic behavior of $\sigma\mathrm{Mex}(n)$, $\sigma_e\mathrm{Mex}(n)$, and $\bar{p}_{\geq 2}(n)$ as $n\to\infty$. In particular, we prove that
		$\sigma\mathrm{Mex}(n)\sim 2\,\sigma_e\mathrm{Mex}(n).$\\
		
		\noindent {\bf \small Keywords:}
		Overpartitions, Minimal excludant, Generating functions, Congruences, Hypergeometric series, Asymptotic formulas.\\
		
		\noindent {\bf \small Mathematics Subject Classification (2020):} 05A17, 11P83 
	\end{abstract}
	\bigskip
	\vspace{0.5em}
\section{Introduction}\label{cs1}
In what follows, we employ the usual notation as in \cite{Gasper}. For any complex
numbers $a$ and $q$ with $|q|<1$, define
\begin{align}
	f_k^m &:= (q^k;q^k)_\infty^m,
	\qquad
	\text{where}
	\qquad
	(a;q)_\infty
	:= \prod_{k=0}^{\infty}(1-aq^k),\nonumber \\
	(a;q)_n &:= \prod_{k=0}^{n-1}(1-aq^k)
	= \frac{(a;q)_\infty}{(aq^n;q)_\infty},
	\quad  n:\quad \text{any integer}. \label{c1.1a}
\end{align}
A partition $\lambda$ of a positive integer $n$ is a finite non-increasing sequence of positive integers $\lambda_1\geq \lambda_2\geq\dots\geq\lambda_k$ whose sum is $n$. The $\lambda_i$ are called the parts of the partition. The number of such partitions of $n$ is denoted by $p(n)$, with convention $p(0):=1$. The generating function for $p(n)$ is given by 
\begin{align*}
	\sum_{n=0}^{\infty}p(n)q^{n}=\frac{1}{(q;q)_\infty}.
\end{align*}
\indent An overpartition of $n$ extends the unrestricted partition to a partition where the first occurrence of a part may be overlined. For example, the eight overpartitions of $3$ are $3,\; \bar{3},\; 2+1,\; \bar{2}+1,\; 2+\bar{1},\; \bar{2}+\bar{1},\; 1+1+1,\; \bar{1}+1+1$. The number of overpartitions of $n$ is denoted by $\bar{p}(n)$. Corteel and Lovejoy \cite{Lovejoy}, noted that the generating function for $\bar{p}(n)$ is
\begin{equation*}
	\sum_{n=0}^{\infty} \bar{p}(n)\, q^n = \frac{f_2}{f_1^{2}}.
\end{equation*}
For a set $S$ of positive integers, a minimal excludant of $S$ is the least positive integer that is not a part of $S$. Andrews and Newman \cite{Newman}, defined the minimal excludant of an integer partition $\lambda$, denoted by $\mathrm{mex}(\lambda)$, as the least positive integer that is not part of $\lambda$. They also introduced sum of minimal excludants over all partitions of $n$, which is defined as
\begin{equation*}
	\sigma{\mathrm{mex}}(n) := \sum_{\lambda \in \mathcal{P}(n)} \mathrm{mex}(\lambda),
\end{equation*}
where $\mathcal{P}(n)$ is the set of all partitions of $n$. For example, we illustrate this in Table \ref{ctab1.1} below, where
\begin{equation*}
	\sigma{\mathrm{mex}}(4)=\sum_{\lambda \in \mathcal{P}(4)}{mex}(\lambda)=9.
\end{equation*}
\begin{center}
	\refstepcounter{table}
	\label{ctab1.1}
	\begin{tabular}{|c|c|c|}
		\hline
		Partition $\lambda$ of $4$ & Parts present & $\mathrm{mex}$\\
		\hline
		$4$ & $\{4\}$ & $1$ \\
		$3+1$ & $\{1,3\}$ & $2$ \\
		$2+2$ & $\{2\}$ & $1$ \\
		$2+1+1$ & $\{1,2\}$ & $3$ \\
		$1+1+1+1$ & $\{1\}$ & $2$ \\
		\hline
	\end{tabular}
	\par\smallskip \textbf{Table \thetable.}\quad Minimal excludant $\mathrm{mex}(\lambda)$ for each partition $\lambda$ of $n=4$.
\end{center}
\indent	Recently, many researchers have investigated minimal excludant variants of partition function, for more details we refer the readers to see \cite{Aricheta,Merca,  Barman,Ray,Stanton, Yee, Bhoria,  Liu, Ray1,Xu}. Baruah et. al \cite{Baruah}, considered sum of minimal excludants restricting the parts by parity. More precisely, they defined the functions  
\begin{equation*}
	\sigma_{o}{\mathrm{mex}}(n) := \sum_{\lambda \in \mathcal{P}(n)} \mathrm{mex}_o(\lambda) \quad \text{and} \quad
	\sigma_{e}{\mathrm{mex}}(n) := \sum_{\lambda \in \mathcal{P}(n)} \mathrm{mex}_e(\lambda),
\end{equation*}
where $\mathrm{mex}_e(\lambda)$ (resp. $\mathrm{mex}_o(\lambda)$) equals $\mathrm{mex}(\lambda)$ if $\mathrm{mex}(\lambda)$ is even (resp. odd), and equals $0$ otherwise.\\
\indent In this paper, we aim to extend the concept of the minimal excludant for overpartitions $\lambda$ of $n$. We now introduce the following definitions for the minimal excludant of the overlined and non-overlined parts of $\lambda$. Throughout the paper we use the notation ${\mathrm{Mex}}$, ${\mathrm{Mex}}_{e}$,  and  ${\mathrm{Mex}}_{o}$ to denote the new class of  minimal excludant, even minimal excludant, and odd minimal excludant for overpartitions, respectively.
\begin{definition}\label{cd1.1}
	The minimal excludant of an overpartition $\lambda$ is the smallest positive
	integer that does not occur as a part of $\lambda$, whether overlined or non-overlined, denoted by ${\mathrm{Mex}}(\lambda)$. For a positive integer $n$, denote the sum of ${\mathrm{Mex}}(\lambda)$ over all overpartitions $\lambda$ of $n$ by $\sigma{\mathrm{Mex}}(n)$:
	\begin{equation}
		\sigma{\mathrm{Mex}}(n) = \sum_{\lambda \in \bar{P}(n)} {\mathrm{Mex}}(\lambda),
	\end{equation}
\end{definition}
\noindent where $\bar{P}(n)$ denotes the set of all overpartitions of $n$. It is easy to note that ${\mathrm{Mex}}(\lambda)$ depends only on values that appear as parts, regardless of the overline. Therefore,
\begin{equation}
	{\sigma}{\mathrm{Mex}}(n) = \sum_{k \geq 1} k \cdot \#\{\lambda \in \bar{P}(n) : {\mathrm{Mex}}(\lambda) = k\}.
\end{equation}
\begin{definition}\label{cd1.2}
	The even minimal excludant of an overpartition $\lambda$ is the smallest
	positive even integer that does not occur among the overlined and non-overlined parts
	of $\lambda$, denoted by ${\mathrm{Mex}}_e(\lambda)$. For all $n \geq 0$:
	\begin{equation}
		{\sigma}_{e}{\mathrm{Mex}}(n) = \sum_{\lambda \in \bar{P}(n)} {\mathrm{Mex}}_e(\lambda),
	\end{equation}
	where ${\mathrm{Mex}}_e(\lambda)$ equals ${\mathrm{Mex}}(\lambda)$ if
	${\mathrm{Mex}}(\lambda)$ is even, and is equal to zero otherwise.
\end{definition}
\begin{definition}\label{cd1.3}\cite{veena}
	The odd minimal excludant of an overpartition $\lambda$ is the smallest
	positive odd integer that does not occur among the overlined and non-overlined parts
	of $\lambda$, denoted by ${\mathrm{Mex}}_o(\lambda)$. For all $n \geq 0$:
	\begin{equation}
		{\sigma}_{o}{\mathrm{Mex}}(n) = \sum_{\lambda \in \bar{P}(n)} {\mathrm{Mex}}_o(\lambda),
	\end{equation}
	where ${\mathrm{Mex}}_o(\lambda)$ equals ${\mathrm{Mex}}(\lambda)$ if
	${\mathrm{Mex}}(\lambda)$ is odd, and is equal to zero otherwise.
\end{definition}
\noindent We illustrate Def. \ref{cd1.1}-\ref{cd1.3}, for  $n = 3$ in Table \ref{ctab1.2} below.
\begin{center}
	\refstepcounter{table}
	\label{ctab1.2}
	\begin{tabular}{|c|c|c|c|c|}
		\hline
		Partition $\lambda$ of 3 & Parts present & ${\mathrm{Mex}}(\lambda)$ & ${\mathrm{Mex}}_e(\lambda)$ & ${\mathrm{Mex}}_o(\lambda)$ \\
		\hline
		$3$ & $\{3\}$ & 1 & 0 & 1\\
		$\bar{3}$ & $\{3\}$ & 1 & 0  & 1\\
		$2+1$ & $\{2, 1\}$ & 3 & 0 & 3 \\
		$\bar{2}+1$ & $\{2, 1\}$ & 3 & 0 & 3 \\
		$2+\bar{1}$ & $\{2, 1\}$ & 3 & 0 & 3 \\
		$\bar{2}+\bar{1}$ & $\{2, 1\}$ & 3 & 0 & 3 \\
		$1+1+1$ & $\{ 1\}$ & 2 & 2 & 0 \\
		$\bar{1}+1+1$ & $\{ 1\}$ & 2 & 2 & 0 \\
		\hline
	\end{tabular}
	\par\smallskip \textbf{Table \thetable.}\quad Minimal excludant over overpartitions of $n=3$.
\end{center}
Let $\bar{p}_{\geq 2}(n)$ denote the number of overpartitions of $n$ with all parts atleast 2. For $0 \leq n \leq 9$, the Table \ref{ctab1.3} below record the values for the combinatorial objects ${\sigma}{\mathrm{Mex}}(n)$, ${\sigma}_{e}{\mathrm{Mex}}(n)$, ${\sigma}_{o}{\mathrm{Mex}}(n)$, and $\bar{p}_{\geq 2}(n)$.
\begin{center}
	\refstepcounter{table}
	\label{ctab1.3}
	\begin{tabular}{|c| c c c c c c c c c c|}
		\hline
		$n$ & 0 & 1 & 2 & 3 & 4 & 5 & 6 & 7 & 8 & 9 \\
		\hline
		${\sigma}{\mathrm{Mex}}(n)$ & 1 & 4 & 6 & 18 & 28 & 50 & 94 & 150 & 238 & 372 \\
		${\sigma}_{e}{\mathrm{Mex}}(n)$ & 0 & 4 & 4 & 4 & 12 & 20 & 60 & 76 & 132 & 196 \\
		${\sigma}_{o}{\mathrm{Mex}}(n)$ & 1 & 0 & 2 & 14 & 16 & 30 & 34 & 74 & 106 & 176 \\
		$\bar{p}_{\geq 2}(n)$ & 1 & 0 & 2 & 2 & 4 & 6 & 10 & 14 & 22 & 32 \\
		\hline
	\end{tabular}
	\par\smallskip \textbf{Table  \thetable.}\quad Minimal excludant $\text{Mex}(\lambda)$ over overpartitions.
\end{center}
In \cite[p. 4, Eq. 1.2.22]{Gasper}, the basic hypergeometric series is defined by
\begin{equation}
	{}_r\phi_s
	\left(
	\begin{matrix}
		a_1,\ldots,a_r\\
		b_1,\ldots,b_s
	\end{matrix}
	;q,z
	\right)
	=
	\sum_{n=0}^{\infty}
	\frac{
		(a_1;q)_n\cdots(a_r;q)_n
	}{
		(q;q)_n
		(b_1;q)_n\cdots(b_s;q)_n
	}
	\left[
	(-1)^nq^{\binom{n}{2}}
	\right]^{1+s-r}
	z^n,
	\label{c4.11}
\end{equation}
where $|z|<1$, $a_1,a_2,\ldots,a_r$ and
$b_1,b_2,\ldots,b_s$ are arbitrary, except that
$(b_j;q)_n\neq0$ for $1\leq j\leq s$, and $(a;q)_n$ is as in
\eqref{c1.1a}. For $|q|<1$, the series on the right-hand
side of ${}_r\phi_s$ converges absolutely for $|z|<1$.
We now set  $r=1,\ s= 1$, $a_1=q,\ \text{and}\ b_1=-q$. Similarly, setting $r=1,\ s= 2$, $a_1=q^2,\ \text{and}\ b_1=-q,\ b_2=-q^2$, finally replacing the base $q\ to\ q^2$, we obtain
\begin{align}
	F_1(z):=&{}_1\phi_1\left(\begin{matrix}	q\\	-q\end{matrix};q,z\right)=\sum_{n=0}^{\infty}\frac{	(-1)^{n}q^{\binom{n}{2}}z^n}{(-q;q)_n},\label{c1.12a} \\
	F_2(z):=&{}_1\phi_2\left(\begin{matrix}q^2\\-q,-q^2\end{matrix};q^2,z\right)=\sum_{n=0}^{\infty}\frac{q^{2n(n-1)}z^n}{(-q;q^2)_n(-q^2;q^2)_n},\label{c1.12b}
\end{align}
respectively.\\
\indent In this sequel, we first establish the generating functions for
$\sigma{\mathrm{Mex}}(n)$ and ${\sigma}_{e}{\mathrm{Mex}}(n)$. Further, in the process, we observe arithmetic properties of $\sigma{\mathrm{Mex}}(n)$ and ${\sigma}_{e}{\mathrm{Mex}}(n)$ modulo powers of $2$, and also obtain relations between basic hypergeometric series with  $\sigma{\mathrm{Mex}}(n)$ and ${\sigma}_{e}{\mathrm{Mex}}(n)$. We now present our main theorems and corollaries.
\begin{theorem}\label{ct1.3}
	For all $n \geq 0$, we have
	\begin{equation*}
		\sum_{n=0}^{\infty} {\sigma}{\mathrm{Mex}}(n)q^{n} = \frac{1}{(q;q)_{\infty}} \sum_{k=0}^{\infty} k \, 2^{(k-1)} q^{\binom{k}{2}} (-q^{k+1};q)_{\infty} \, (1 - q^k).
	\end{equation*}
\end{theorem}
\begin{theorem}\label{ct1.4}
	For all $n \geq 0$, we have
	\begin{equation*}
		\sum_{n=0}^{\infty} {\sigma}_{e}{\mathrm{Mex}}(n)q^{n} = \frac{1}{(q;q)_{\infty}} \sum_{j=0}^{\infty} j \, 2^{2j} q^{\binom{2j}{2}} (-q^{2j+1};q)_{\infty} \, (1 - q^{2j}).
	\end{equation*}
\end{theorem}
Using Theorems~\ref{ct1.3} and \ref{ct1.4}, we establish relations between ${\sigma}{\mathrm{Mex}}(n)$, ${\sigma}_{e}{\mathrm{Mex}}(n)$,
and $\bar{p}_{\geq 2}(n)$ in the following Cor. \ref{ct1.5}-\ref{ct1.8}.
\begin{corollary}\label{ct1.5}
	We have
	\begin{align}
		{\sigma}{\mathrm{Mex}}(n) &\equiv 0 \pmod{2} \quad \forall\ n \geq 1 \label{c1.1} \\
		{\sigma}{\mathrm{Mex}}(n) &\equiv \bar{p}_{\geq 2}(n) \pmod{2^2} \quad \forall\, n \geq 0 \label{c1.2}\\
		{\sigma}_{e}{\mathrm{Mex}}(n) &\equiv 0 \pmod{2^2} \quad \forall\ n \geq 0 . \label{c1.5}
	\end{align}
\end{corollary}
\begin{corollary}\label{ct1.7}
	For all $n \geq 1$, we have
	\begin{align}
		{\sigma}_{e}{\mathrm{Mex}}(n) &\equiv \bar{p}_{\geq 2}(n) \pmod{2} \label{c1.3} \\
		{\sigma}_{e}{\mathrm{Mex}}(n) &\equiv \bar{p}_{\geq 2}(n) \pmod{2^2} \quad \text{if and only if } n \text{ is a perfect square.} \label{c1.4}
	\end{align}
\end{corollary}
\begin{corollary}\label{ct1.8}
	For all $n \geq 0$, we have
	\begin{equation*}
		{\sigma}_{e}{\mathrm{Mex}}(n) \equiv {\sigma}{\mathrm{Mex}}(n) - \bar{p}_{\geq 2}(n) \pmod{2^2}.
	\end{equation*}
\end{corollary}
\noindent Now we obtain relation between $\sigma{\mathrm{Mex}}(n)$ and ${\sigma}_{e}{\mathrm{Mex}}(n)$ with basic hypergeometric series in the following Theorems \ref{ct1.9c}-\ref{ct1.9b}.
\begin{theorem}\label{ct1.9c}
	For $|q|<1$, we have
	\begin{equation*}
		\sum_{n=0}^{\infty}\sigma\mathrm{Mex}(n)q^n=\frac{(-q;q)_\infty}
		{2(q;q)_\infty}\lim_{z\to {-2}} z\left(F_1^{'}(z)-qF_1^{'}(qz)\right),
	\end{equation*}
	where $	F_1(z)$	is as defined in \eqref{c1.12a}.	
\end{theorem}
\begin{theorem}\label{ct1.9b}
	For $|q|<1$, we have 
	\begin{equation*}
		\sum_{n=0}^{\infty}\sigma_{e}\mathrm{Mex}(n)q^n=\frac{(-q;q)_\infty}
		{(q;q)_\infty}\lim_{z\to {4q}} z \left(F_2^{'}(z)-q^2F_2^{'}(q^2z)\right),
	\end{equation*}
	where $	F_2(z)$ is as defined in \eqref{c1.12b}.
\end{theorem}
Further, our overarching goal in this work is to study ${\sigma}{\mathrm{Mex}}(n)$ and $ {\sigma}_{e}{\mathrm{Mex}}(n)$ from an asymptotic perspective. With this in mind, we prove the following theorems.
\begin{theorem}\label{ct1.9a}
	As $n \to \infty$, we have
	\begin{equation*}
		{\sigma}{\mathrm{Mex}}(n) \sim \frac{\sqrt{2}}{8} n^{-3/4} e^{\pi \sqrt{n}}. 
	\end{equation*}
\end{theorem}
\begin{theorem}\label{ct1.10}
	As $n \to \infty$, we have
	\begin{equation*}
		{\sigma}_{e}{\mathrm{Mex}}(n) \sim \frac{1}{8\sqrt{2}} n^{-3/4} e^{\pi \sqrt{n}}. 
	\end{equation*}
\end{theorem}
\begin{corollary}\label{cc1.11}
	As $n \to \infty$,
	\begin{equation*}
		{\sigma}{\mathrm{Mex}}(n) \sim 2 {\sigma}_{e}{\mathrm{Mex}}(n). \label{eq:1.7}
	\end{equation*}
\end{corollary}
\begin{theorem}\label{ct1.12}
	As $n \to \infty$, we have
	\begin{equation*}
		\bar{p}_{\geq 2}(n) \sim \frac{\pi}{32}\, n^{-3/2}\, e^{\pi\sqrt{n}}.
	\end{equation*}
\end{theorem}
In Section \ref{cs2}, we first recall the necessary theta function identities and establish the generating functions of ${\sigma}{\mathrm{Mex}}(n)$ and ${\sigma}_{e}{\mathrm{Mex}}(n)$, which helps us to prove our Cor. \ref{ct1.5}-\ref{ct1.8}. Section \ref{cs5} is devoted to providing relations between the basic hypergeometric series with $\sigma\mathrm{Mex}(n)$ and $\sigma_e\mathrm{Mex}(n)$.  In Section \ref{cs4}, we establish the asymptotic formulas for ${\sigma}{\mathrm{Mex}}(n)$, ${\sigma}_{e}{\mathrm{Mex}}(n)$, and $\bar{p}_{\geq 2}(n)$. All the proof techniques used herein are elementary and relying on classical $q$-series identities and generating function manipulations. 
\section{Generating Functions and Arithmetic Congruences}\label{cs2}
Ramanujan's theta function $f(a,b)$ and two special cases of it, $\varphi(q)$ and $\psi(q)$ along with certain identities satisfied by them, will be beneficial in the proof of Theorems \ref{ct1.5}-\ref{ct1.8}. Ramanujan's general theta function, see \cite{Berndt,Berndt.1}, which is defined for $|ab|<1$, is given by
\begin{equation*}
	f(a,b):=\sum_{n=-\infty}^{\infty}
	a^{\frac{n(n+1)}{2}}b^{\frac{n(n-1)}{2}}.
\end{equation*}
And its special cases are 
\begin{align}
	\varphi(q)	&:=\sum_{n=-\infty}^{\infty}q^{n^2}
	=\frac{f_2^5}{f_1^2f_4^2}, \label{c3.1}\\
	\psi(q)&:=\sum_{n=0}^{\infty}q^{n(n+1)/2}
	=\frac{f_2^2}{f_1}.\label{c3.2}
\end{align}
We obtain \eqref{c3.1} and \eqref{c3.2} using well-known Jacobi triple product identity, which is given by \cite[p. 35]{Berndt}
\begin{equation*}
	f(a,b)=(-a,ab)_{\infty}(-b,ab)_{\infty}(ab,ab)_{\infty}.
\end{equation*}
Lastly, we will utilize the following result,
which, at its core, is an easy consequence of binomial theorem and the divisibility properties of various binomial coefficients.
\begin{lemma}\label{cl3.1}
	For all positive integers $k$ and $m$, we have
	\begin{align*}
		f_m^{2^{k}}\equiv f_{2m}^{2^{k-1}} \pmod{2^k}.
	\end{align*}
\end{lemma}
In very next lemma, we obtain congruences modulo $2^{2}$ for  the overpartition $\overline{p}_{\ge2}(n)$. 
\begin{lemma}\label{cl3.2}
	For all $n\ge 1$, we have 
	\begin{align*}
		\sum_{n=0}^{\infty}\bar{p}_{\ge2}(n)q^n\equiv
		\begin{cases}
			0 \pmod{2^{2}}, & \text{if } n \text{ is a perfect square} \\[4pt]
			2 \pmod{2^{2}}, & \text{otherwise}.
		\end{cases}
	\end{align*}
	\begin{proof} We have
		\begin{align}
			\sum_{n=0}^{\infty}\bar{p}_{\ge2}(n)q^n =\prod_{n=2}^{\infty}\dfrac{1+q^n}{1-q^n}
			&=\dfrac{f_2}{f_1^2}\cdot \dfrac{(1-q)}{(1+q)}\label{c3.3a}\\
			&=\dfrac{f_2}{f_1^2} \left(1+2\sum_{n=1}^{\infty}(-1)^nq^n\right)\nonumber\\
			&= \dfrac{f_2}{f_1^2} \left(1+2\sum_{n=1}^{\infty}q^{n^2}+2\sum_{\substack{n=1\\ n\not=k^2}}^{\infty}q^n\right)\nonumber\\
			&=\dfrac{f_2}{f_1^2}\left(\varphi(q)+2\sum_{\substack{n=1\\ n\not=k^2}}^{\infty}q^n\right) \nonumber\\
			&\equiv\left(\dfrac{f_2}{f_1^2}\varphi(q)+2\dfrac{f_2}{f_1^2}\sum_{\substack{n=1\\ n\not=k^2}}^{\infty}q^{n}\right)\pmod{2^2}.\label{c3.3}
		\end{align}
		Employing \eqref{c3.1} and Lemma \ref{cl3.1} in \eqref{c3.3}, we obtain
		\begin{align*}
			\sum_{n=0}^{\infty}\bar{p}_{\ge2}(n)q^n
			&\equiv \left(\dfrac{f_{2}^{6}}{{f_{1}^{4}}{f_{4}^{2}}} + 2\dfrac{f_2}{f_1^2}\sum_{\substack{n=1\\ n\not=k^2}}^{\infty}q^{n}\right)\pmod{2^2}\nonumber\\
			&\equiv 1+2\sum_{\substack{n=1\\ n\not=k^2}}^{\infty}q^n\pmod{2^2}.
		\end{align*}
		This completes the proof of Lemma \ref{cl4.2}. 
	\end{proof}
\end{lemma}
\noindent Next, with all of these tools, we are now in a position to first establish the generating function for ${\sigma}{\mathrm{Mex}}(n)$ and  ${\sigma}_{e}{\mathrm{Mex}}(n)$ employing Def. \ref{cd1.1}-\ref{cd1.2}, respectively.
\begin{proof}[Proof of Theorem \ref{ct1.3}]
	Define	
	\begin{align}
		F_{k}(q)& := \underbrace{\prod_{m=1}^{(k-1)} \frac{2q^m}{1-q^m}}_{\text{values } 1,\ldots,(k-1) \text{ forced}} \cdot \underbrace{1}_{\text{value } k \text{ absent}} \cdot \underbrace{\prod_{m=k+1}^{\infty} \frac{1+q^m}{1-q^m}}_{\text{values} > k\text{ free}} \nonumber\\
		&=2^{(k-1)}\prod_{m=1}^{(k-1)}q^{m}\prod_{m=1}^{(k-1)}\frac{1}{1-q^{m}}\cdot\prod_{m=k+1}^{\infty}(1+q^{m})\prod_{m=k+1}^{\infty}\frac{1}{1-q^{m}} \nonumber\\
		&=\frac{2^{(k-1)} \, q^{\binom{k}{2}}}{(q;q)_{(k-1)}} \cdot  (-q^{k+1};q)_\infty\cdot\frac{1}{(q^{k+1};q)_\infty} \nonumber \\
		&=\frac{2^{(k-1)} \, q^{\binom{k}{2}}}{(q;q)_{(k-1)}} \cdot \frac{(-q^{k+1};q)_\infty \, (q;q)_{(k-1)}{(1-q^{k})}}{(q;q)_\infty}\nonumber\\
		&=\frac{2^{(k-1)} \, q^{\binom{k}{2}}(-q^{k+1};q)_\infty \,{(1-q^{k})}}{(q;q)_\infty}.\label{c2.1}
	\end{align}
	Therefore, on using Def. \ref{cd1.1} in \eqref{c2.1}, we obtain
	\begin{align*}
		\sum_{n=0}^{\infty} {{\sigma}{\mathrm{Mex}}}(n)q^{n}=&\sum_{k=0}^{\infty}{k}\cdot{F_{k}(q)}\\
		=&\frac{1}{(q;q)_\infty}\sum_{k=0}^{\infty}{k \,2^{k-1} \, q^{\binom{k}{2}}(-q^{k+1};q)_\infty \,{(1-q^{k})}}.
	\end{align*}
\end{proof}
\noindent And similarly, we now move to obtain the generating function for $\sigma_{e}{\mathrm{Mex}}(n)$.
\begin{proof}[Proof of Theorem \ref{ct1.4}]
	In Theorem \ref{ct1.3}, substituting $k=2j$ in the identity \eqref{c2.1}, we obtain	
	\begin{align}
		F_{2j}(q):=\frac{2^{2j} \, q^{\binom{2j}{2}}(-q^{2j+1};q)_\infty \,{(1-q^{2j})}}{2(q;q)_\infty}.\label{c2.3}	
	\end{align}
	Therefore, on using Def. \ref{cd1.2} in \eqref{c2.3}, we obtain
	\begin{align*}
		\sum_{n=0}^{\infty} {{\sigma}_{e}{\mathrm{Mex}}}(n)q^{n}=&\sum_{j=0}^{\infty}{2j}\cdot{F_{2j}(q)}\\
		=&\frac{1}{(q;q)_\infty}\sum_{j=0}^{\infty}{j \,2^{2j} \, q^{\binom{2j}{2}}(-q^{2j+1};q)_\infty \,{(1-q^{2j})}}.
	\end{align*}
\end{proof}
\begin{remark}\label{cr2.2}
	For $n=3$, the eight overpartitions of $3$ with their ${\mathrm{Mex}}$ values are shown in the table below. 
	\begin{center}
		\begin{tabular}{|c|c|c|c|c|}
			\hline
			Overpartition & Parts present & ${\mathrm{Mex}}(\lambda)$ & ${\mathrm{Mex}}_e(\lambda)$ & ${\mathrm{Mex}}_{o}(\lambda)$ \\
			\hline
			$3,\ \bar{3}$ & $\{3\}$ & $1$ & $0$ & $1$\\
			$2+1,\ \bar{2}+1,\ 2+\bar{1},\ \bar{2}+\bar{1}$ & $\{1,2\}$ & $3$ & $0$ & $3$\\
			$1+1+1,\ \bar{1}+1+1$ & $\{1\}$ & $2$ & $2$ & $0$ \\
			\hline
		\end{tabular}
	\end{center}
	It is easy to verify that, $\sigma{\mathrm{Mex}}(3) = 2(1) + 4(3) + 2(2) =18 $, $\sigma_{e}{\mathrm{Mex}}(3) = 2(0) + 4(0) + 2(2) = 4$, and $\sigma_{o}{\mathrm{Mex}}(3) = 2(1) + 4(3) + 2(0) = 14$.
\end{remark}
Finally, with the above prerequisites, we are now equipped to prove our Cor. \ref{ct1.5}-\ref{ct1.8}.
\begin{proof}[Proof of Corollary \ref{ct1.5}]
	From Theorem \ref{ct1.3}, we have
	\begin{equation}
		k \cdot 2^{k-1} \equiv 0 \pmod{2} \qquad \forall\ \ k \geq 2.
		\label{c3.4}
	\end{equation}
	Since, for modulo $2$, only $k=1$ term of the sum in Theorem \ref{ct1.3} survives. Hence, we obtain
	\begin{equation}
		\sum_{n=0}^{\infty} {\sigma}{\mathrm{Mex}}(n) q^n\equiv \frac{(1-q)(-q^2;q)_{\infty}}{(q;q)_{\infty}}
		= \frac{(-q^2;q)_{\infty}}{(q^2;q)_{\infty}}= \sum_{n=0}^{\infty} \bar{p}_{\geq 2}(n) q^n \pmod{2}.
		\label{c3.5}
	\end{equation}
	On the other hand, since $(-q^2;q)_{\infty} \equiv (q^2;q)_{\infty} \pmod 2$, we have
	\begin{equation}
		\sum_{n=0}^{\infty} \bar{p}_{\geq 2}(n) q^n \equiv 1 \pmod{2}. \label{c3.6}
	\end{equation}
	Therefore, $\bar{p}_{\geq 2}(n) \equiv 0 \pmod 2$ for all $n \geq 1$. Combining this with \eqref{c3.4} completes the proof of \eqref{c1.1}.
	Again, by Lemma \ref{cl3.2}, comparing the coefficients of $q^n$ on both sides of the congruence in \eqref{c3.5}, we complete the proof of \eqref{c1.2}. Similarly, thanks to Theorem \ref{ct1.4}, since $2^{2j} \equiv 0 \pmod{2^2}$, this immediately implies the proof of \eqref{c1.5}. This completes the proof of Cor. \ref{ct1.5}.
\end{proof}
\begin{proof}[Proof of Corollary \ref{ct1.7}]
	Thanks to \eqref{c1.5} and \eqref{c3.6}, we have
	\begin{equation}
		\sum_{n=0}^{\infty} {\sigma}_{e}{\mathrm{Mex}}(n) q^n \equiv \prod_{n=2}^{\infty} \frac{1+q^n}{1-q^n} \pmod{2} = \sum_{n=1}^{\infty} \bar{p}_{\geq 2}(n) q^n.
		\label{c3.7}
	\end{equation}
	Comparing the coefficients of $q^n$ on both sides of congruence \eqref{c3.7}, completes the proof of \eqref{c1.3}. Similarly, again using identity \ref{c1.5} and Lemma \ref{cl3.2}, we complete the proof of Cor. \ref{ct1.7}.
\end{proof}
\begin{proof}[Proof of Corollary \ref{ct1.8}]
	Using Cor. \ref{ct1.5} in Lemma \ref{cl3.2}, we have
	\begin{equation*}
		{\sigma}_{e}{\mathrm{Mex}}(n) \equiv {\sigma}{\mathrm{Mex}}(n) - \bar{p}_{\geq 2}(n) \equiv
		\begin{cases}
			0 \pmod{2^2}, & \text{if } n \text{ is a perfect square} \\	2 \pmod{2^2}, & \text{otherwise}.
		\end{cases}
	\end{equation*}
	This completes the proof of Cor. \ref{ct1.8}.
\end{proof}
\section{Relation with Basic Hypergeometric Series}\label{cs5}
In this section, we strengthen the connection between   $\sigma\mathrm{Mex}(n)$, $\sigma_e\mathrm{Mex}(n)$ with $	F_1(z)$ and $F_2(z)$, respectively.
\begin{proof}[Proof of Theorem \ref{ct1.9c} ]
	Thanks to the Theorem \ref{ct1.3} and  identity $(-q^{k+1};q)_\infty= \frac{(-q;q)_\infty}{(-q;q)_k}$, we have
	\begin{equation}
		\sum_{n=0}^{\infty}\sigma{\mathrm{Mex}}(n)q^n
		=
		\frac{(-q;q)_\infty}{2(q;q)_\infty}
		\sum_{k=1}^{\infty}
		\frac{
			k2^k(1-q^{k})q^{\binom{k}{2}}
		}{(-q;q)_k}.
		\label{c2.1f}
	\end{equation}
	We next take a significant step forward using \eqref{c1.12a} and note the following
	\begin{equation}
		F_1(z)-F_1(qz)=\sum_{k=0}^{\infty}\frac{(-1)^k(1-q^{k})q^{\binom{k}{2}}z^k}{(-q;q)_k}.\label{c2.1c}
	\end{equation}
	Apply the operator $z\frac{d}{dz}$ term wise and let $z\to -2$ in  \eqref{c2.1c}, we obtain
	\begin{equation}
		\lim_{z\to -2} z\left(F_1^{'}(z)-qF_1^{'}(qz)\right)	=\sum_{k=0}^{\infty}\frac{k2^k(1-q^{k})q^{\binom{k}{2}}}{(-q;q)_k}. \label{c4.14}
	\end{equation}
	Employing the identity \eqref{c4.14} in \eqref{c2.1f}, completes the proof of Theorem \ref{ct1.9c}.
\end{proof}
\begin{proof}[Proof of Theorem \ref{ct1.9b}]
	Thanks to the Theorem \ref{ct1.4} and identities $	(-q^{2j+1};q)_\infty= \frac{(-q;q)_\infty}{(-q;q^2)_j(-q^2;q^2)_j}$ and  $\binom{2j}{2}=2j(j-1)+j$, we obtain
	\begin{equation}
		\sum_{n=0}^{\infty}\sigma_{e}\mathrm{Mex}(n)q^n
		=
		\frac{(-q;q)_\infty}{(q;q)_\infty}
		\sum_{j=1}^{\infty}
		\frac{
			j4^{j}(1-q^{2j})q^{2j(j-1)}q^j
		}{
			(-q;q^2)_j(-q^2;q^2)_j
		}.
		\label{c2.1g}
	\end{equation}
	Similarly, from \eqref{c1.12b}, we have
	\begin{equation}
		F_2(z)-F_2(q^2z)=\sum_{j=0}^{\infty}\frac{(1-q^{2j})q^{2j(j-1)}z^j}{(-q;q^2)_j(-q^2;q^2)_j}.\label{c2.1e}
	\end{equation}
	Apply the operator $z\frac{d}{dz}$ term wise and let $z\to 4q$ in  \eqref{c2.1g}, we obtain
	\begin{align}
		\lim_{z\to {4q}}z\left\{F_2^{'}(z)-q^2F_2^{'}(q^2z)\right\}
		=&\sum_{j=1}^{\infty}\frac{j(1-q^{2j})q^{2j(j-1)}(4q)^j}{(-q;q^2)_j(-q^2;q^2)_j}.\label{c4.15}
	\end{align}
	Employing the identity \eqref{c4.15} in \eqref{c2.1g}, completes the proof of Theorem \ref{ct1.9b}.
\end{proof}
\section{Asymptotic Formula}\label{cs4}
In this section, we establish the asymptotic formulas for $\sigma{\mathrm{Mex}}(n)$, $\sigma_{e}{\mathrm{Mex}}(n)$, and $\overline{p}_{\geq 2}(n)$, with the help of following asymptotic results by Ingham \cite{ingham} for the coefficients of power series.
\begin{theorem}\label{ct4.1}
	Let $A(q)=\sum_{n=0}^{\infty}a(n)q^n$ be a power series with radius of convergence equal to $1$. Suppose that $\{a(n)\}_{n\geq0}$ is a weakly increasing sequence of non-negative real numbers. If there exist constants $\alpha,\beta\in\mathbb{R}$ and $C>0$ such that
	\begin{equation*}
		A(e^{-t})\sim\alpha t^\beta e^{C/t},\qquad t\to0^+
	\end{equation*}
	then
	\begin{equation*}
		a(n)\sim\frac{\alpha}{2\sqrt{\pi}}\,C^{\frac{2\beta+1}{4}}\,n^{-\frac{2\beta+3}{4}}\,e^{2\sqrt{Cn}},\qquad n\to\infty.	
	\end{equation*}
\end{theorem}
\begin{lemma}\label{cl4.2}
	Let $N=N(t)\to\infty$ as $t\to0^+$ with $Nt\to0$ and $N^2t=O(1)$. Then 
	\begin{equation*}
		\prod_{k=1}^{N}(1+e^{-kt})\sim 2^{N} e^{-N^2t/4}.
	\end{equation*}
\end{lemma}
\begin{proof}
	Define
	\begin{equation*}
		\beta(x):=\log\left(1+e^{-x}\right),\qquad x\geq0.
	\end{equation*}
	Then
	\begin{equation}
		\sum_{k=1}^{N}\beta(x)=\log\prod_{k=1}^{N}\left(1+e^{-x}\right).\label{c4.1b}
	\end{equation}
	Since $\beta(0)=\log 2$,\ $\beta'(0)=-\frac{1}{2}$, and  $\beta''(0)=\frac{1}{4}.$ Therefore, by Taylor's series expansion,
	\begin{equation*}
		\beta(x)=\log 2-\frac{x}{2}+\frac{x^2}{8}+O(x^4), \ \textit{as} \ x\to0.
	\end{equation*}
	Since $Nt\to0$, we have $kt\to0$ uniformly for $1\leq k\leq N$. Hence,
	\begin{equation*}
		\beta(kt)=\log 2-\frac{kt}{2}+\frac{k^2t^2}{8}+O(k^4t^4).
	\end{equation*}
	Therefore, we have
	\begin{align*}
		\sum_{k=1}^{N}\beta(kt)=N\log 2-\frac{t}{2}\sum_{k=1}^{N}k+\frac{t^2}{8}\sum_{k=1}^{N}k^2+O\left(t^4\sum_{k=1}^{N}k^4\right).
	\end{align*}
	It is easy to note, $-\frac{t}{2}\sum_{k=1}^{N}k
	=-\frac{N^2t}{4}-\frac{Nt}{4}$, $t^2\sum_{k=1}^{N}k^2=O(N^3t^2)=O\bigl((N^2t)(Nt)\bigr)=o(1)$, and  $t^4\sum_{k=1}^{N}k^4=O(N^5t^4)=O\bigl((N^2t)^2(Nt)\bigr)=o(1).$
	Since $Nt\to0$, it follows that $-\frac{Nt}{4}=o(1)$.
	Therefore on employing \eqref{c4.1b}, we obtain
	\begin{equation*}
		\log\prod_{k=1}^{N}\left(1+e^{-kt}\right)=N\log 2-\frac{N^2t}{4}+o(1).
	\end{equation*}
	Hence,
	\begin{equation*}
		\prod_{k=1}^{N}(1+e^{-kt})= 2^{N} e^{-\frac{N^2t}{4} (1+o(1))}.
	\end{equation*}
	Therefore,
	\begin{equation*}
		\prod_{k=1}^{N}(1+e^{-kt})\sim 2^{N} e^{-\frac{N^2t}{4}},	
	\end{equation*}
	which completes the proof.
\end{proof}
\noindent We now move to the proofs of Theorems \ref{ct1.9a}-\ref{ct1.10}.
\begin{proof}[Proof of Theorem \ref{ct1.9a}]
	Define
	\begin{equation*}
		E(q):=\sum_{n=0}^{\infty}\sigma{{\mathrm{Mex}}}(n)q^n.
	\end{equation*}
	On substituting $q=e^{-t}$ in above identity and then employing Theorem \ref{ct1.3}, we obtain
	\begin{equation}
		E(e^{-t})=\sum_{N}{N}\cdot{F_{N}(t)}.\label{c4.1}
	\end{equation}
	Further, using \eqref{c2.1} with $f_1(t):=(e^{-t};e^{-t})_\infty$ and $f_2(t):=(e^{-2t};e^{-2t})_\infty$, we obtain
	\begin{equation*}
		F_{N}(t)=\frac{1}{f_1(t)}\cdot2^{N-1}e^{-N(N-1)t/2}
		(1-e^{-Nt})\cdot\frac{f_2(t)}{f_1(t)}\cdot\frac{1}{\displaystyle\prod_{k=1}^{N}(1+e^{-kt})}.
	\end{equation*}
	Hence, on employing Lemma \ref{cl4.2}, we obtain
	\begin{equation}
		{N}\cdot{F_{N}(t)} \sim \frac{f_2(t)}{2f_1^{2}(t)}Ne^{-(N^2t/4-Nt/2)}(1-e^{-Nt}).\label{c4.2}
	\end{equation}
	For integers $N$ and $u\geq 0$, set $N=\frac{u}{\sqrt{t}}$, as $Nt\to0$ and $N^{2}t\to u^{2}$. We have $1-e^{-Nt}\sim Nt=u{\sqrt{t}}.$ Therefore,
	\begin{equation*}
		Ne^{-(N^2t/4-Nt/2)}(1-e^{-Nt})\sim u^2e^{-u^2/4}.
	\end{equation*}
	By Riemann sum, we have $\int_0^\infty u^2 e^{-u^2/4}\,du = 2\sqrt{\pi}$. Hence,
	\begin{align*}
		\sum_{N} Ne^{-(N^2t/4-Nt/2)}(1-e^{-Nt}) \sim \frac{1}{\sqrt{t}}
		\int_0^\infty u^2e^{-u^2/4}\,du =\frac{2\sqrt{\pi}}{\sqrt{t}}.
	\end{align*}
	Consequently, using \eqref{c4.1} and \eqref{c4.2}, we obtain
	\begin{equation*}
		\sum_{N}{N}\cdot{F_{N}(t)}\sim\frac{\sqrt{\pi}}{\sqrt{t}}\frac{f_2(t)}{f_1^2(t)}.
	\end{equation*}
	Since, $f_1(t)\sim\sqrt{\frac{2\pi}{t}}e^{-\frac{\pi^2}{6t}}$ and $f_2(t)\sim\sqrt{\frac{\pi}{t}}e^{-\frac{\pi^2}{12t}}$, above asymptotic formula simplifies to
	\begin{equation*}
		E(e^{-t})=\sum_{N}{N}\cdot{F_{N}(t)}\sim\frac{1}{2}e^{\frac{\pi^2}{4t}}.
	\end{equation*}
	On setting $	\alpha=\frac12,
	\ \beta=0, \ \text{and} \ C=\frac{\pi^2}{4}$ in Theorem  \ref{ct4.1}, we obtain
	\begin{align*}
		\sigma{\mathrm{Mex}}(n)\sim\frac{1}{4\sqrt{\pi}}\left(\frac{\pi^2}{4}\right)^{1/4}n^{-3/4}e^{\pi\sqrt{n}}.
	\end{align*}
	This completes the proof of Theorem \ref{ct1.9a}.
\end{proof}
\begin{proof}[Proof of Theorem \ref{ct1.10}]
	Define
	\begin{equation*}
		E_e(q):=\sum_{n=0}^{\infty}\sigma_{e}{\mathrm{Mex}}(n)q^n.
	\end{equation*}
	We now first set $q=e^{-t}$ in Theorem \ref{ct1.4} and the replace $N = 2j$ in the resulting identity to obtain
	\begin{equation*}
		E_e(e^{-t})=\frac{1}{f_1(t)}\sum_{N}{N}2^{N-1}e^{-N(N-1)t/2}
		(1-e^{-Nt})(-e^{-(N+1)t};e^{-t})_\infty,
	\end{equation*}
	which is same as identity \eqref{c4.1}. Therefore, for $u\geq 0$, we obtain
	\begin{equation*}
		Ne^{-(N^2t/4-Nt/2)}(1-e^{-Nt})\sim {u^2}e^{-u^2/4}.
	\end{equation*}
	Since $N$ is even,  $\Delta u=2\sqrt{t}$. By Riemann sum, we have $\int_0^\infty u^2 e^{-u^2/4}\,du = 2\sqrt{\pi}$. Hence,
	\begin{align*}
		\sum_{N} Ne^{-(N^2t/4-Nt/2)}(1-e^{-Nt}) \sim \frac{1}{2\sqrt{t}}
		\int_0^\infty {u^2}e^{-u^2/4}\,du =\frac{\sqrt{\pi}}{\sqrt{t}}.
	\end{align*}
	Therefore,
	\begin{equation}
		E_e(e^{-t})\sim\frac14e^{\frac{\pi^2}{4t}}. \label{c4.6}
	\end{equation}
	Now, applying Theorem \ref{ct4.1} in \eqref{c4.6}. We complete the Proof of Theorem \ref{ct1.10}.
\end{proof}
\begin{proof}[Proof of Corollary \ref{cc1.11}]
	From Theorems \ref{ct1.9a} and \ref{ct1.10}, we conclude, $\sigma{\mathrm{Mex}}(n)$ and $ \sigma_{e}{\mathrm{Mex}}(n)$ have same exponential growth rate of $e^{\pi\sqrt{n}}$ with a difference of factor 2.
\end{proof}
\begin{proof}[Proof of Theorem \ref{ct1.12}]
	From \eqref{c3.3a}, we have 
	\begin{equation*}
		\sum_{n=0}^{\infty} \bar{p}_{\geq 2}(n) q^n = \frac{(-q^2;q)_{\infty}}{(q^2;q)_{\infty}}.
		\label{c4.1a}
	\end{equation*}
	Define
	\begin{equation*}
		G(q) := \frac{1-q}{1+q}\cdot \frac{f_2}{f_1^{2}}.
	\end{equation*}
	For $q = e^{-t}$, it is easy to observe, $\frac{1-q}{1+q} \sim \frac{t}{2} \ \text{as} \ t\to0^{+}$.\\
	
	\noindent Finally in $G(e^{-t})$, we employ $f_1(t)\sim\sqrt{\frac{2\pi}{t}}e^{-\frac{\pi^2}{6t}}$, $f_2(t)\sim\sqrt{\frac{\pi}{t}}e^{-\frac{\pi^2}{12t}}$, and Theorem \ref{ct4.1} to complete the proof of Theorem \ref{ct1.12}.	
\end{proof}	

	\bigskip
\bigskip

\noindent
Department of Mathematics\\
Ramanujan School of Mathematical Sciences\\
Pondicherry University\\
Puducherry- 605 014, India.\\

\noindent Email: \texttt{iamdipikasarkar@pondiuni.ac.in}

\noindent Email: \texttt{tthejithamp@pondiuni.ac.in}

\noindent	Email: \texttt{dr.fathima.sn@pondiuni.ac.in} (\Letter)
\end{document}